\documentclass[a4paper,11pt]{amsart}

\usepackage{anysize} \marginsize{0.7in}{0.7in}{1in}{1in}
\usepackage{comment}
\usepackage{xcolor}
\usepackage{amsmath}
\usepackage{mathtools}
\usepackage[all]{xy}
\usepackage[utf8]{inputenc}
\usepackage{varioref}
\usepackage{ulem}
\usepackage{amsfonts}
\usepackage{amssymb}
\usepackage{bbm}
\usepackage{esint}
\usepackage{graphicx}
\usepackage{tikz}
\usepackage{empheq}
\usepackage{enumitem}
\usepackage{tikz-cd}
\usetikzlibrary{matrix,arrows,decorations.pathmorphing}
\usepackage{mathrsfs}
\usepackage[hypertexnames=false,backref=page,pdftex,
 	pdfpagemode=UseNone,
 	breaklinks=true,
 	extension=pdf,
 	colorlinks=true,
 	linkcolor=blue,
 	citecolor=red,
 	urlcolor=blue,
 ]{hyperref}

\usepackage{soul}

\definecolor{brickred}{rgb}{0.8, 0.25, 0.33}

\newcommand{\bbP}{{\mathbb P}}
\newcommand{\bbR}{{\mathbb R}}

\newcommand{\Bl}{\operatorname{Bl}}

\newcommand{\ord}{{\mathrm{ord}}}

\theoremstyle{plain}

\newtheorem{theorem}[subsection]{Theorem}
\newtheorem{definition}[subsection]{Definition}

\newtheorem{lemma}[subsection]{Lemma}

\newtheorem{proposition}[subsection]{Proposition}

\theoremstyle{remark}

\newtheorem{remark}[subsection]{Remark}

\title[K-stability]{On K-stability of $\mathbb P^3$ blown up along a $(2,3)$ complete intersection}

\author{Luca Giovenzana}
\address[Luca Giovenzana]{Department of Mathematical Sciences, Loughborough University,
Loughborough,
Leicestershire,
LE11 3TU, United Kingdom}
\email{l.giovenzana@lboro.ac.uk}

\author{Tiago Duarte Guerreiro}
\address[Tiago Duarte Guerreiro]{Department of Mathematical Sciences, University of Essex,
Wivenhoe Park, Colchester CO4 3SQ, United Kingdom}
\email{t.duarteguerreiro@essex.ac.uk}

\author{Nivedita Viswanathan}
\address[Nivedita Viswanathan]{School of Mathematical Sciences, University of Nottingham, Nottingham, Nottinghamshire, NG9 2RD, United Kingdom}
\email{Nivedita.Viswanathan@nottingham.ac.uk}

\makeatletter
\@namedef{subjclassname@2020}{%
 \textup{2020} Mathematics Subject Classification}
\makeatother



\thanks{Luca Giovenzana was supported by Engineering and Physical Sciences
Research Council (EPSRC) New Investigator Award EP/V005545/1 "Mirror
Symmetry for Fibrations and Degenerations".}
\thanks{Tiago Duarte Guerreiro was partially supported by Engineering and Physical Sciences
Research Council (EPSRC)  EP/V048619/1 and EP/V055399/1.}
\thanks{Nivedita Viswanathan was supported by the Engineering and Physical Sciences
Research Council (EPSRC) New Horizons Grant EP/V048619/1.}

\usepackage{framed}

\begin{document}

\begin{abstract}	
We prove K-stability of every smooth member of the family 2.15 of the Mukai-Mori classification.
\end{abstract}

\maketitle

\section{Introduction}\label{section introduction}
\thispagestyle{empty}


The existence of a K\"ahler-Einstein metric on a compact manifold $X$ is a foundational problem in complex geometry. In the seminal series of papers \cite{CDS1,CDS2,CDS3, Tian}, the authors prove that such an existence has an algebro-geometric characterisation, known as K-polystability, and hence solving the famous Yau-Tian-Donaldson conjecture. From then on, a great deal of work has been carried out to verify K-stability of Fano manifolds. Of particular importance is the work of Abban and Zhuang \cite{AZ21} where the authors introduce a new powerful inductive framework. These new techniques have been most notably used in \cite{ACCFKMGSSV} where the 105 families of smooth Fano 3-folds have been analysed. Despite the very thorough investigation, the K-stability of a few smooth members of certain families is still yet to be described. In recent times, there has been much progress made in this regard. See \cite{LiuXu19, Kentorank3deg28, denisova2022kstability,cheltsov2022kstable,CheltsovPark2022,belousov2022kstability, Liu-Rank2Deg14,cheltsov2023kstable, CheltsovFujitaKishimotoPark,Malbon}. Our paper aims at giving yet another step in the full understanding of the K-stability of smooth Fano 3-folds by looking at every smooth member of family 2.15. Our main result is a complete answer in that case:

\begin{theorem}[Main Theorem. See Theorem \ref{thm:main}]
Every smooth member of the Fano family 2.15, which is the blow-up of $\mathbb P^3$ in a curve given by the intersection of a quadric and a cubic, is K-stable.
\end{theorem}

\subsection{Structure of the paper:}
In Section~\ref{section: AB Theory} we recall the preliminaries and the result from Abban-Zhuang theory that we use to prove the result.
In Section~\ref{section: computations}, after a brief presentation of the smooth members of the family, we show the main theorem by estimating the local stability threshold $\delta_p$. The computations are split according to the position of the point $p$. Particular care has to be taken when the unique quadric containing the blown-up curve is singular.

\subsection*{Acknowledgements:} We are grateful for the research environment that we were provided with during EDGE Days 2021 and 2022 as well as  CIRM-FBK for sponsoring the SinG School at the University of Trento. We also thank Ivan Cheltsov for suggesting the problem to us and his constant guidance. Moreover, we want to thank Hamid Abban, Erroxe Etxabarri Alberdi, Franco Giovenzana, Kento Fujita and Alan Thompson for numerous discussions and the interest shown.

\section{Abban-Zhuang theory }\label{section: AB Theory}

In this section we recall the definition of K-stability and the main results used in order to prove Theorem~\ref{thm:main}.

\begin{definition}
Let $\Delta$ be an effective $\mathbb{Q}$-divisor on a normal projective variety $X$ for which $K_X+\Delta$ is $\mathbb{Q}$-Cartier. We say that $(X,\Delta)$ is a \textbf{log Fano pair} if $(X, \Delta)$ is klt and $-(K_X + \Delta)$ is ample. If $\Delta = 0$, we call $(X, 0)$ a Fano variety and denote it by $X$.
 \end{definition}

We recall the notion of stability threshold (or $\delta$-invariant) introduced in \cite{kentoodaka}.
 \begin{definition}
 Let $(X,\Delta)$ be a log Fano pair, and let $f : Y \rightarrow X $ be a projective birational morphism such that $Y$ is normal and let $E$ be a prime divisor on $Y$. Let $L$ be an ample $\mathbb Q$-Cartier divisor on $X$. We set,
\[
A_{X,\Delta}(E)=1+\mathrm{ord}_E\big(K_Y-f^*(K_X+\Delta)\big), \quad S_L(X)=\frac{1}{L^n}\int_0^{\infty} \mathrm{vol}(f^*(L)-uE)du.
\]
We define the \textbf{stability threshold} as
\[
\delta(X,\Delta;L)=\inf_{E/X}\frac{A_{X,\Delta}(E)}{S_L(X)}
\]
where the infimum runs over all prime divisors over $X$. For a point $p \in X$, we define the \textbf{local stability threshold} as
\[
\delta_p(X,\Delta;L)=\inf_{\substack{E/X \\ p \in C_X(E)}}\frac{A_{X,\Delta}(E)}{S_L(X)}
\]
where the infimum runs over all prime divisors over $X$ whose centres on $X$ contain $p$.
 \end{definition}
 
It is proved in \cite{kentoodaka,blumjonsson} that the following equivalence holds,
\[
\delta(X)>1 \iff X\,\, \text{is K-stable}.
\]
We will, in fact, take this to be our definition of K-stability of a Fano variety.  Moreover,
\[
\delta(X,\Delta;L)= \inf_{p \in X}\delta_p(X,\Delta;L).
\]
\begin{definition}[{\cite[Definition~1.1]{kentoplt}}]
    Let $\Delta$ be an effective $\mathbb Q$-divisor on $X$ and $(X, \Delta)$ be a klt pair. A prime divisor $Y$ over $X$ is said to be of \textbf{plt-type} over $(X, \Delta)$ if there is a projective birational morphism $\mu : \tilde X \rightarrow X$ with $Y \subset \tilde X$ such that $-Y$ is a $\mu$-ample $\mathbb Q$-Cartier divisor on $\tilde X$ for which $(\tilde X,\tilde \Delta+ Y)$ is a plt pair where 
 the $\mathbb Q$-divisor $\tilde \Delta$ is defined by\[
    K_{\tilde X}+ \tilde \Delta + (1-A_{X,\Delta}(Y))Y=\mu^*(K_X+\Delta). 
    \]
 \end{definition}

 \begin{remark}
     The morphism $\mu$ is completely determined by $Y$ and it is called the \textbf{plt-blowup} associated to $Y$.
 \end{remark}
In the following we study K-(semi)stability of certain Fano 3-folds $X$. We do this by employing the Abban-Zhuang theory developed in \cite{AZ21} to estimate the local stability threshold $\delta_p$ for every point in $X$. We recall the main results we need by referring to the book \cite{ACCFKMGSSV}.

Given a smooth Fano threefold $X$, so that, in particular Nef($X$)=Mov($X$), and a point $p\in X$ we consider flags $p\in Z\subset Y \subset X$ where:
\begin{itemize}
\item $Y$ is an irreducible surface with at most Du Val singularities;

\item $Z$ is a non-singular curve such that $(Y,Z)$ is plt.
\end{itemize}
We denote by $\Delta_Z$ the different of the log pair $(Y,Z)$.

For $u\in \bbR$, we consider the divisor class $-K_X-uY$ and we denote by $\tau=\tau (u)$ its pseudoeffective threshold, i.e. the largest number for which $-K_X-uY$ is pseudoffective. For $u\in[0,\tau]$, let $P(u)$ (respectively $N(u)$) be the positive (respectively negative) part of its Zariski decomposition.
Since $Y\not\subset\mathrm{Supp} (N(u))$ we can consider the restriction $N(u)\vert_Y$ and define $N'_Y(u)$ to be its part not supported on $Z$, i.e. $N'_Y(u)$ is the effective $\mathbb R$-divisor such that $Z\not\subset \mathrm{Supp}(N'_Y(u))$ defined by:
\[
N_Y(u)= d(u)Z + N'_Y(u)
\]
where $d(u):=$ord$_Z(N(u)\vert_Y)$.

We consider then  for every $u\in [0,\tau]$ the restriction $P(u)\vert_Y$ and denote by $t(u)$ the pseudoeffective threshold of the divisor $P(u)\vert_Y-vZ$, by $P(u,v)$ and $N(u,v)$ the positive and negative part of its Zariski decomposition.
Let $V^Y_{\bullet,\bullet}$ and $W_{\bullet,\bullet,\bullet}^{Y,Z}$ be the multigraded linear series defined in \cite[Page~57]{ACCFKMGSSV}.

Finally we can state the main tool we use to estimate the local $\delta_p$-invariant:
\begin{theorem}\cite[Theorem~1.112]{ACCFKMGSSV}\label{delta estimate}
\begin{align*}
\delta_p(X)\geq \mathrm{min} \left\{ \frac{1-\mathrm{ord}_p\Delta_Z}{S(W_{\bullet,\bullet,\bullet}^{Y,Z};p)},\  \frac{1}{S(V^Y_{\bullet,\bullet};Z)},\ \frac{1}{S_X(Y)}\right\}
\end{align*} where 
\begin{align}\label{surface-to-curve}
S(V^Y_{\bullet,\bullet};Z)=\frac{3}{(-K_X)^3}\int_0^{\tau}(P(u)^2 \cdot Y) \cdot \mathrm{ord}_Z(N(u)\vert_Y)du+\frac{3}{(-K_X)^3}\int_0^{\tau}\int_0^{\infty}\mathrm{vol}(P(u)\vert_Y-vZ)dvdu,
\end{align}
and 
\begin{align}\label{curve-to-point}
S(W_{\bullet,\bullet,\bullet}^{Y,Z};p)=\frac{3}{(-K_X)^3}\int_0^{\tau}\int_0^{t(u)}(P(u,v) \cdot Z)^2dvdu+F_p(W_{\bullet,\bullet,\bullet}^{Y,Z}),
\end{align}
 with 
\begin{align}\label{Fp}
F_p(W_{\bullet,\bullet,\bullet}^{Y,Z})=\frac{6}{(-K_X)^3}\int_0^{\tau}\int_0^{t(u)}(P(u,v) \cdot Z)\cdot \mathrm{ord}_p(N_Y'(u)\vert_Z+N(u,v)\vert_Z)dvdu.
\end{align}
\end{theorem}

The theorem above admits a slight generalization which allows to consider not only flags of varieties in $X$, but also over $X$. In particular,
let $X$ and $Y$ be as above, in order to estimate $\delta_p$ for $p\in Y$ it turns out to be useful to consider curves over $Y$. For this, let $\sigma: \widetilde Y\to Y$ be a plt blow-up of $Y$ in $p$ and denote by $\widetilde Z$ its exceptional divisor.
We consider the linear system $\sigma^*(P(u)\vert_Y)-v \widetilde Z$ and denote by $\tilde t(u)$ its pseudoeffective threshold, i.e.
\begin{align*}
\tilde t(u)=\mathrm{max}\{v\in \mathbb{R}_{\geq 0}\ : \ \sigma^{*}(P(u)\vert_{Y})-v\widetilde Z\ \mathrm{is\ pseudoeffective}\}.
\end{align*}

For every $v\in [0,\tilde t(u)]$ we denote by $\widetilde P(u,v)$ and $\widetilde N(u,v)$ the positive and negative part of its Zariski decomposition. We also denote by $N'\vert_{\widetilde Y}(u)$ the strict transform of the divisor $N(u)\vert_Y$. 

\begin{theorem}\label{blow-up-of-surface-formula}\cite[Remark~1.113]{ACCFKMGSSV}
\begin{align*}
\delta_p(X)\geq \mathrm{min} \left\{ \mathrm{min}_{q\in\tilde Z}\frac{1-\mathrm{ord}_q\Delta_{\tilde{Z}}}{S(W_{\bullet,\bullet,\bullet}^{Y,\tilde{Z}};q)},\ 
\frac{A_Y(\tilde{Z})}{S(V^Y_{\bullet,\bullet};\tilde{Z})},\
\frac{1}{S_X(Y)}\right\}
\end{align*}
where:
\begin{equation}
\begin{aligned}\label{Rmk1.7.32-surface-to-curve}
S(V^Y_{\bullet,\bullet};\widetilde Z)=&\frac{3}{(-K_X)^3}\int_0^{\tau}(P(u)^2 \cdot Y) \cdot \mathrm{ord}_{\widetilde{Z}}(\sigma^*( N(u)\vert_Y))du\ +\\
&\frac{3}{(-K_X)^3}\int_0^{\tau}\int_0^{\infty}\mathrm{vol}\left(\sigma^*\left(P(u)\vert_Y\right)-v\widetilde Z\right)dvdu,
\end{aligned}
\end{equation}
and 
\begin{align}\label{Rmk1.7.32-curve-to-point}
S(W_{\bullet,\bullet,\bullet}^{Y,\tilde{Z}};q)=\frac{3}{(-K_X)^3}\int_0^{\tau}\int_0^{\tilde{t}(u)}(\tilde{P}(u,v) \cdot \tilde{Z})^2dvdu
&+F_q(W_{\bullet,\bullet,\bullet}^{Y,\tilde{Z}}),
\end{align}
with 
\begin{align}\label{Rmk1.7.32-FP}
F_q(W_{\bullet,\bullet,\bullet}^{Y,\tilde Z})=\frac{6}{(-K_X)^3}\int_0^{\tau}\int_0^{\tilde t(u)}(\tilde P(u,v) \cdot \tilde Z)\cdot \mathrm{ord}_q(N_{\tilde{Y}}'(u)\vert_{\tilde Z}+\tilde N(u,v)\vert_{\tilde Z})dvdu.
\end{align}
\end{theorem}\bigskip\bigskip

\section{ K-stability of the family 2.15}\label{section: computations}
 We briefly review the geometry of a smooth Fano threefold in the family 2.15. Each smooth member is a threefold of Picard number 2 obtained as the blow-up of $\bbP^3$ in a (2,3)-complete intersection, see \cite[Section~4.4]{ACCFKMGSSV} and references therein. 

 Let $ \mathscr C\subset\bbP^3$ be the complete intersection of a quadric $Q = (f_2=0)$ and a cubic $S_3 = (f_3=0)$. We are interested in the K-stability of the blow-up $X:=\Bl_{\mathscr C} \bbP^3$. We stress the fact that the quadric $Q$ can be either smooth or a quadric cone. Let $\alpha:X\to \bbP^3$ be the projection, $E$ the exceptional divisor and $\widetilde{Q}$ the strict transform of $Q$. 
The linear system of cubics vanishing along $\mathscr C$ gives a rational map:
\begin{align*}
\varphi\colon \bbP^3 & \dashrightarrow \bbP^4&\\
                  [x:y:z:w] &\mapsto [xf_2:yf_2:zf_2:wf_2:f_3].
\end{align*}
with indeterminacy locus $\mathscr C$. The blow-up $X$ is a resolution of indeterminacy of $\varphi$ fitting in the diagram
\begin{align*}
\xymatrix{
&E\subseteq X\supseteq \widetilde{Q} \ar[dr]^{\beta} \ar[dl]_{\alpha} \\
\bbP^3 \ar@{-->}[rr] && V_3\subseteq \bbP^4
}
\end{align*}
where $\beta$ contracts $\widetilde Q$ to a point and maps $X$ to a cubic threefold $V_3$ singular at the point $\beta(\widetilde Q)={[0:0:0:0:1]}$.

We denote by $H\in~$NS$(X)$ the pullback of the line bundle $\mathcal O_{\mathbb P^3}(1)$ along $\alpha$. The Neron-Severi group of $X$ is generated by $H$ and $E$ and its anti-canonical divisor is given by
\begin{align*}
-K_X= 4H-E = 2H + \widetilde Q = 2\widetilde Q + E,
\end{align*}
 where we used the equality $\widetilde{Q}=2H-E$.
We denote by $f_1\in N_1(X)$ the class of the fiber of the restriction $\alpha\vert_E\colon E\to \mathscr C$ and by $f_2\in N_1(X)$ the class of a ruling of $\widetilde Q$ so that the Mori cone is $\overline{NE}(X) = \mathbb R_{\geq 0}f_1 + \mathbb R_{\geq 0}f_2$. The intersection numbers are as follows:
\begin{align*}
    &E\cdot f_1 = \widetilde{Q}\cdot f_2 = -1,\quad E\cdot f_2 = 3\\
    &H\cdot f_2 = \widetilde{Q}\cdot f_1 = 1, \quad  H\cdot f_1=0,\\
    &H^3=1,\quad H\cdot E^2=-6,\quad H^2 \cdot E=0,\quad \mathrm{ and }\\
    &E^3 = - \deg N_{\mathscr C|\mathbb{P}^3} =-2g+2+K_{\mathbb{P}^3} \cdot \mathscr C = -30.
\end{align*}

\subsection{ Estimate of \texorpdfstring{$\delta_p$ for $p$  in $\widetilde{Q}$ when $Q$ is a smooth quadric}.}
In this section we estimate the K-stability threshold $\delta_p$ for a point $p\in \widetilde Q$ by applying Theorem~\ref{delta estimate} to a specific flag.

\begin{proposition}\label{delta: p in Q}
If $p$ is a point in $\widetilde Q$ and not in $E$, then
\begin{align*}
    \delta_p(X) = \frac{44}{37}
\end{align*}
and it is computed by the divisor $\tilde Q$ in $X$. If $p\in E\cap \widetilde Q$, then
\begin{align*}
    \delta_p(X) \geq \frac{8}{7}.
\end{align*}
\end{proposition}

\begin{proof}
Given a point $p\in \widetilde Q$, we consider the flag
\begin{align*}
    p \in L \subset \widetilde Q \subset X
\end{align*}
where $L$ is a line of $\widetilde Q$ through $p$ which is not tangent to the curve $E\cap \widetilde Q$ at $p$, or equivalently, whose image under the map $\alpha$ is not tangent to $\mathscr C$ at $\alpha(p)$.

We start by computing $S_X(\widetilde Q)$. For this, we consider the linear system $K_X-u\widetilde Q = E + (2-u)\widetilde Q$ for $u\in \mathbb R$. Clearly its pseudoeffective threshold is $\tau=2$. The Zariski decomposition is given by: \footnote{Since the Zariski decomposition is defined to be $P(u)$ and $N(u)$, here it is confusing to use $P_{\widetilde Q}(u)$. Would suggest sticking to $P(u)$. }
\[P(u)=
\begin{cases}
(4-2u)H+(u-1)E & \text{if}\ u\in [0,1],\\
(4-2u)H & \text{if}\ u\in [1,2],
\end{cases}
\quad \text{and} \quad 
N(u)=
\begin{cases}
0 & \text{if}\ u\in [0,1],\\
(u-1)E & \text{if}\ u\in [1,2].
\end{cases}\]

Therefore the volume can be computed to be:
\begin{align*}
\mathrm{vol}(-K_X-u\tilde{Q})= \left(P(u)\right)^3=
\begin{cases}
22-6u-6u^2-2u^3& \text{if}\ u\in [0,1],\\
64-96u+48u^2-8u^3 & \text{if}\ u\in [1,2].
\end{cases}
\end{align*}

Hence we get:
\begin{align}\label{Flag1:3fold-to-surface}
    S_X(\widetilde Q) = \frac{1}{(-K_X)^3}\int_{0}^{\tau(\tilde{Q})}\mathrm{vol}(-K_X-u\tilde{Q})du = \frac{37}{44}.
\end{align}

We move on to compute the value $S(V^{\widetilde Q}_{\bullet,\bullet};L)$. For this, let $\ell_1, \ell_2$ the classes of the rulings of $\widetilde Q$ so that the class of $L$ is $\ell_1$, we consider for $v\in \mathbb R_{\geq 0}$ the linear system:
\begin{align*}
P(u)\vert_{\widetilde{Q}}-vL=
\begin{cases}
(1+u-v)\ell_1+(1+u)\ell_2 & \text{if}\ u\in [0,1],\\
(4-2u-v)\ell_1+(4-2u)\ell_2& \text{if}\ u\in [1,2].
\end{cases}
\end{align*}
The nefness and bigness of the above linear system is readily checked and its Zariski decomposition is given by:
\begin{align*}
P(u,v)=
\begin{cases}
(1+u-v)\ell_1+(1+u)\ell_2 & \text{if}\ u\in [0,1],\ v\in [0,1+u]\\
(4-2u-v)\ell_1+(4-2u)\ell_2& \text{if}\ u\in [1,2],\ v\in [0,4-2u],
\end{cases}N(u,v)=
\begin{cases}
0\\
0.
\end{cases}
\end{align*}
Hence
\begin{align*}
\mathrm{vol}(P(u)\vert_{\widetilde{Q}}-vL)=
\begin{cases}
2(1+u-v)(1+u) & \text{if}\ u\in [0,1],\ v\in [0,1+u]\\
4(4-2u-v)(2-u)& \text{if}\ u\in [1,2],\ v\in [0,4-2u].
\end{cases}
\end{align*}

We note that the restriction of the divisor $E$ to $\widetilde Q$ consists of an irreducible curve which is isomorphically mapped to $\mathscr C$ by the blow-up morphism $\alpha$. In particular, we see that $E\vert_{\widetilde Q}$ has no support on $L$ and the negative part $N(u)$ does not contribute in the formula \eqref{surface-to-curve} and we get:
\begin{align} \label{Flag1:surface-to-curve}
    S(V^{\widetilde Q}_{\bullet,\bullet};L) = \frac{69}{88}.
\end{align}

We move on now to compute $S(W_{\bullet,\bullet,\bullet}^{\widetilde Q,L};p)$. 

If the point $p\in \widetilde Q\setminus E$, then the order of $E\vert_{\widetilde Q}$ at $p$ is trivial, hence the value $F_p(W_{\bullet,\bullet,\bullet}^{\widetilde Q,L})$ of \eqref{Fp} is zero. A direct computation gives the value of \eqref{curve-to-point}:
\begin{align}\label{Flag1:curve-to-point}
S(W_{\bullet,\bullet,\bullet}^{\widetilde Q,L};p) = \frac{69}{88}.
\end{align}

On the other hand, if the point $p$ is in $\widetilde Q\cap E$ the value $F_p$ in \eqref{Fp} is not trivial. First of all we notice that $L$ is not contained in $E\vert_{\widetilde Q}$ so we have $N(u)=N'_{\widetilde Q}(u)$. Secondly, since in the choice of the flag we assumed that $L$ intersects $E\cap \widetilde Q$ transversely we have $\mathrm{ord}_p(N'_{\widetilde Q}(u)\vert_L) = u-1 \mbox{ if } u\in[1,2].$
For the value in \eqref{Fp} we therefore get:
\begin{align}\label{Flag1:Fp}
    F_p = \frac{1}{11}.
\end{align}

If $p\not\in E$, the values $S_X(\widetilde Q)$, $S(V^{\widetilde Q}_{\bullet,\bullet};L)$ and $S(W_{\bullet,\bullet,\bullet}^{\widetilde Q,L};p)$ are computed in the formulas \eqref{Flag1:3fold-to-surface}, \eqref{Flag1:surface-to-curve} and \eqref{Flag1:curve-to-point}, so that:
\begin{align*}
  \frac{44}{37} = \frac{1}{S_X(\widetilde Q)} \geq \delta_p(X)\geq \mathrm{min}\bigg\{\frac{44}{37},\ \frac{88}{69},\ \frac{88}{69}\bigg\} = \frac{44}{37}.
\end{align*}

If the point $p$ is in $E$, the value $S(W_{\bullet,\bullet,\bullet}^{\widetilde Q,L};p)$ is obtained by summing up also $F_p$, which is computed in \eqref{Flag1:Fp} and one gets:
\begin{align*}
    \delta_p(X)\geq \mathrm{min}\bigg\{\frac{44}{37},\ \frac{88}{69},\ \frac{8}{7}\bigg\} = \frac{8}{7}.
\end{align*}
This concludes the proof.
\end{proof}

\subsection{ Estimate of \texorpdfstring{$\delta_p$ for $p$ in $\widetilde{Q}$ when $Q$ is a quadric cone}.}

We divide the computations in two separate cases: These are when $p$ is the vertex of the quadric cone or $p$ is away from it.

\subsubsection{\texorpdfstring{$p$ is the vertex of the quadric cone}.}

Let $\pi \colon \hat{X} \rightarrow X$ be the blowup of $X$ at $p$ with exceptional divisor $G \simeq \mathbb{P}^2$. Let $\hat{Q}$ be the strict transform of $\tilde Q$ in $\hat{X}$. Since $\hat{Q}=\pi^*\tilde Q-2G$ and $-K_X=2\tilde Q+E$, we have
\begin{align}\label{eq: canonical vertex blow-up}
\pi^*(-K_X)-uG = 2\hat Q +\hat E + (4-u)G
\end{align}
 where $\hat E \simeq E$ is the strict transform of $E$ in $\hat X$.
 
 \begin{lemma}\label{lemma:bigness of divisor II}
 The pseudo-effective threshold $\tau$ of the linear system $\pi^*(-K_X)-uG$ is $\tau = 4$.
\end{lemma}
 
 \begin{proof}
From Equation~\eqref{eq: canonical vertex blow-up} we clearly we have that $\tau\geq 4$. In order to prove the equality it is enough to show that the divisor $2\hat Q +\hat E$ is not big. For this, let $\gamma\colon \hat X\to \Bl_{\alpha(p)}\mathbb P^3$ be the divisorial contraction of $\hat E$. Since the pushforward of a big divisor along a birational morphism is big, in order to show the claim it is enough to show that $\gamma_{*}\hat Q$ is not big. For this, notice that $\Bl_{\alpha(p)}\mathbb P^{3}$ is the resolution of indeterminacy of the projection from $\alpha(p)$ and is a conic bundle $h\colon\Bl_{\alpha(p)}\mathbb P^{3}\to \mathbb P^2$ which contracts $\gamma(\hat Q)$ to a conic. In particular $\gamma(\hat Q) \equiv h^*\mathcal O_{\mathbb P^2}(2)$ is not big. The claim is proven.
\end{proof}
 
 Let $l, \, f_G$ and $f_E$ be the ruling of $\hat Q$, the class in $\mathrm{Pic}(G)$ of a line of $G$ and a fibre of $E$, respectively. We have the following intersection numbers

 \begin{center}
\begin{tabular}{ c|ccc } 

 & $l$ & $f_G$ & $f_E$ \\ 
 \hline
 $\hat Q$ & $-3$ & $2$ & $1$ \\ 
 $G$ & $1$ & $-1$ & $0$ \\ 
 $\hat E$ & $3$ & $0$ & $-1$ \\ 
 
\end{tabular}
\end{center}
Moreover,
\begin{align*}
 \hat Q^2\cdot \hat E = -6, \quad  \hat Q\cdot G^{2} = -2,\quad  \hat Q^2 \cdot G = 4, \quad G^2\cdot \hat E = G \cdot \hat E^2 = 0,\\
 \hat E^{3} = -30\quad  \hat Q\cdot \hat{E}^{2} = 18,\quad \hat Q^{3} = -6, \quad \hat Q \cdot \hat E \cdot G = 0,\quad G^{3} = 1. 
\end{align*}

\begin{proposition} \label{prop: p vertex}
If $p$ is the vertex of the quadric cone $\tilde Q$, then
\begin{align*}
    \delta_p(X) = \frac{11}{10},
\end{align*}
and it is computed by the exceptional divisor $G$ corresponding to the ordinary blowup of $X$ at $p$.
\end{proposition}

\begin{proof}

     By \cite[Corollary 4.18 (2)]{Kentorank3deg28}, we have
\begin{equation} \label{eq:AZmain}
 \frac{A_X(G)}{S_X(G)} \geq \delta_p(X) \geq  \min \bigg\{\frac{A_X(G)}{S_X(G)},
\inf_{\substack{q \in G}} \delta_q(G,\Delta_G;V^{G}_{\bullet,\bullet}) \bigg\}.
\end{equation}
We compute first $ \frac{A_X(G)}{S_X(G)}$ and then show this is the bound given by the right hand side of the second inequality of \eqref{eq:AZmain}. Let $P(u)$ and $N(u)$ be the positive and negative part of $\pi^{*}(-K_X) - uG$. We have:
\begin{align*}
P(u)=
\begin{cases}
2\hat Q +\hat E + (4-u)G & \text{if}\ u\in[0,1],\\
\frac{7-u}{3}\hat Q +\hat E + (4-u)G & \text{if}\ u\in[1,4],
\end{cases}
\mbox{ and }
N(u)=
\begin{cases}
0 & \text{if}\ u\in [0,1],\\
\frac{(u-1)}{3}\hat Q & \text{if}\ u\in [1,4],
\end{cases}
\end{align*}

A direct computation gives:
 \[
 \frac{A_X(G)}{S_X(G)}=\frac{11}{10}.
 \]
We now compute $\inf_{\substack{q \in G}} \delta_q(G,\Delta_G;V^{G}_{\bullet,\bullet})$.

\begin{itemize}
    \item Suppose $q \not \in \hat Q \vert_G$.
\end{itemize}
For every such point we choose a flag $q \in L \subset G$, where $L$ is a line in $G$. Then,  by \cite[Theorem~3.2]{AZ21}
\begin{align*}
\delta_q (G,\Delta_G; W^G_{\bullet,\bullet})
\geq \mathrm{min} \left\{ \frac{1}{S(W_{\bullet,\bullet}^G;L)},\frac{1-\mathrm{ord}_q{\Delta_L}}{S(W_{\bullet,\bullet,\bullet}^{G,L};q)} \right\}.
\end{align*}
 Let $P(u,v)$ and $N(u,v)$ be the positive and negative part of $P(u)\vert_G-vL$. These are given by 
\begin{align*}
    P(u,v)=
    \begin{cases}
    (u-v)L &\text{if}\ u\in [0,1],\ v\in [0,u],\\
    \Big(\frac{2+u}{3}-v\Big)L &\text{if}\ u\in [1,4],\ v\in [0, \frac{2+u}{3}],
    \end{cases}\quad
    \text{ and}\quad  N(u,v) = 0.
\end{align*}
Notice that $\mathrm{ord}_L(N(u)\vert_{G})=0$ since $\hat Q\vert_G$ is not supported on $L$ and $\mathrm{ord}_q(N'_G(u)\vert_{L}+N(u,v)\vert_L)=0$ since $q \not \in \hat Q\vert_G$. Hence, 
\[
\frac{1}{S(W_{\bullet,\bullet}^G;L)}=\frac{1-\ord_q{\Delta_L}}{S(W_{\bullet,\bullet,\bullet}^{G,L};q)}=\frac{44}{23}.
\]

\begin{itemize}
    \item Suppose $q \in \hat Q\vert_G$.
\end{itemize}
We denote by $\eta : \hat G \rightarrow G$ the $(1,2)$-weighted blowup of $q$ with exceptional divisor $F \simeq \mathbb{P}(1,2)$. By \cite[Corollary~4.18 (1)]{Kentorank3deg28}, we have
\begin{equation} \label{eq:extra-blowup}
\delta_q(G,\Delta_{G};W^{\hat G}_{\bullet,\bullet}) \geq  \min \bigg\{\frac{A_G(F)}{S(V^{\hat G}_{\bullet, \bullet};F)},
\inf_{\substack{q' \in F \\ \eta(q')=q}} \frac{A_{F,\Delta_F}(q')}{S(W^{\hat {G},F}_{\bullet,\bullet,\bullet};q')}\bigg\}.
\end{equation}

The surface $\hat G$ has an $A_1$ singular point $q_0$ lying on $F$. Denote by $C$ the conic $\hat Q\vert_G$ and by $\ell_T$ the line tangent to $C$ at $q$. Their strict transforms $\widetilde C$ and $\widetilde{\ell_T}$ intersect $F$ at a regular point of $\hat G$. We have 
 \begin{align*}
 \widetilde{C}=\eta^*C-2F,\qquad  \widetilde{\ell_T}=\eta^*\ell_T-2F,\  \mbox{ and }\\
\widetilde{\ell_T}^2 = -1, \quad \tilde{C}^2=2, \quad F^2=-\frac{1}{2}, \quad \widetilde{\ell_T}\cdot F = 1.
\end{align*}
We consider the linear system
\[
\eta^*(P(u)\vert_G)-vF=\begin{cases}
u\widetilde{\ell_T}+(2u-v)F & \text{if}\ u \in [0,1]  ,\\
\frac{2+u}{3}\widetilde{\ell_T}+\Big(\frac{2}{3}(2+u)-v \Big)F  & \text{if}\ u \in [1,4].
\end{cases}
\]
Then, its Zariski decomposition has positive part
\begin{align*}
\tilde P(u,v)=
\begin{cases}
u\widetilde{\ell_T} + (2u-v)F & \text{if } u\in [0,1]\    v\in [0,u]\\
(2u-v)(\widetilde{\ell_T}+F)        & \text{if } u\in [0,1]\     v\in [u,2u] \\
\frac{2+u}{3}\widetilde{\ell_T} + \Big(\frac{4+2u}{3}-v\Big)F   & \text{if } u\in [1,4]\ v\in [0,\frac{2+u}{3}] \\
\frac{4+2u}{3}(\widetilde{\ell_T}+F)        & \text{if } u\in [1,4]\ v\in [\frac{2+u}{3},\frac{4+2u}{3}].
\end{cases}
\end{align*}
and negative part
\begin{align*}
\tilde N(u,v)=
\begin{cases}
0 & \text{if } u\in [0,1]    v\in [0,u]\\
(v-u)\widetilde{\ell_T}      & \text{if } u\in [0,1]    v\in [u, 2u] \\
0   & \text{if } u\in [1,4]\ v\in [0,\frac{2+u}{3}] \\
(v-\frac{2+u}{3})\widetilde{\ell_T}     & \text{if } u\in [1,4]\ v\in [\frac{2+u}{3},\frac{4+2u}{3}].
\end{cases}
\end{align*}

Notice that
\begin{align*}
\ord_F(\eta^*N(u)\vert_G)&=
\begin{cases}
    0 & \text{if}\ u \in [0,1]\\
    \ord_F \Big(\frac{u-1}{3}\eta^*C\Big) & \text{if}\ u \in [1,4]
\end{cases} 
=\begin{cases}
0 & \text{if}\ u \in [0,1]  ,\\
\frac{2}{3}(u-1) & \text{if}\ u \in [1,4].
\end{cases}
\end{align*} 

A direct computation gives
\[
\frac{A_G(F)}{S(V^{\hat G}_{\bullet,\bullet};F)}=\frac{11}{10}.
\]

We now compute the second term in formula \eqref{eq:extra-blowup}. For $u \in [0,1]$,
\begin{align*}
\mathrm{ord}_{q'}(\eta^*(N'_{\tilde G}(u)\vert_{F}+N(u,v)\vert_F))&=\mathrm{ord}_{q'}(\eta^*N(u,v)\vert_F)\\ 
&=\mathrm{ord}_{q'}((v-u)\widetilde{\ell_T}\vert_F)\\
&=\begin{cases}
0 & \text{if}\ q' \not \in \widetilde{\ell_T}  ,\\
v-u  & \text{otherwise}.
\end{cases}
\end{align*}
On the other hand, for $u \in [1,4]$,
\begin{align*}
\mathrm{ord}_{q'}(\eta^*(N'_{\tilde G}(u)\vert_{F}+N(u,v)\vert_F))
&=\mathrm{ord}_{q'}\Bigg( \frac{u-1}{3}\widetilde C\vert_F + \Big(v-\frac{2+u}{3} \Big)\widetilde{\ell_T}\vert_F\Bigg)\\
&=\begin{cases}
0 & \text{if}\ q' \not \in \widetilde{\ell_T} \cup \widetilde C  ,\\
\frac{u-1}{3} & \text{if}\ q'  \in \widetilde C  ,\\
v-\frac{2+u}{3}  & \text{if}\ q' \in \widetilde{\ell_T}.
\end{cases}
\end{align*}
Then, 
\[
S(W^{\tilde{G},F}_{\bullet,\bullet,\bullet};q')= 
\begin{cases}
\frac{23}{88} & \text{if}\ q' \not \in \widetilde{\ell_T} \cup \widetilde C  ,\\
\frac{37}{44} & \text{if}\ q'  \in \widetilde C  ,\\
\frac{23}{44}  & \text{if}\ q' \in \widetilde{\ell_T}.
\end{cases}
\]
Moreover, $A_{F,\Delta_F}(q')=1$ for every $q' \in \tilde{F}$ except when $q'$ is the $A_1$ singularity introduced by $\eta$, in which case it is $\frac{1}{2}$. Hence, 
\begin{align*}
\inf_{\substack{q' \in F \\ \eta(q')=q}} \frac{A_{F,\Delta_F}(q')}{S(W^{\tilde{G},F}_{\bullet,\bullet,\bullet};q')}&= \min \Big\{\frac{1}{23/88},\frac{1/2}{23/88},\frac{1}{23/44},\frac{1}{37/44} \Big\}\\
&=\min\bigg\{\frac{88}{23},\frac{88}{46},\frac{44}{23},\frac{37}{44} \bigg\}\\
&=\frac{44}{37}.
\end{align*}
Therefore,
\begin{equation*} 
\delta_q(G,\Delta_{G};W^{\tilde G}_{\bullet,\bullet}) \geq \min\bigg\{\frac{11}{10},   
\frac{44}{37}\bigg\}=\frac{11}{10}
\end{equation*}
for $q \in C$.

Putting together the cases, $q \not \in C$ and $q \in C$, we have indeed,
\begin{equation*} 
\delta_q(G,\Delta_{G};W^{\tilde G}_{\bullet,\bullet}) \geq \min\bigg\{\frac{11}{10},   
\frac{44}{23}\bigg\}=\frac{11}{10}.
\end{equation*}
Hence,
\[
\delta_p(X)\geq \frac{11}{10}
\]
and the claim follows.
\end{proof}

\subsubsection{The point $p$ is away from the vertex of the quadric cone.}\phantom.

Let $p$ be any point in $\tilde{Q}$ such that $\alpha(p)$ is not the vertex of $Q$. We consider the general hyperplane section $H$ of $\mathbb P^3$ containing $\alpha(p)$ and its strict transform $S$ in $X$. Then $S$ is isomorphic to the blow-up of $H$ in the six points $p_1,...,p_6$ given by $Q\cap \mathscr C$, which lie on the conic $C=Q\cap H$.

We consider the blow-up $\sigma\colon \widetilde S \to S$ in the point $p$ with exceptional divisor $F$. We denote by $\widetilde C$ the strict transform of $C$ in $\widetilde S$, by $E_1,...,E_6$ the curves lying over the points $p_1,..., p_6$ and by $L_j$ the strict transform of the line through the points $\alpha(p)$ and $p_j$ for $j=1,...,6$.

\begin{proposition} \label{prop: p not vertex}
Assume that $Q$ is a quadric cone. Let $p\in X$ be a point such that $\alpha(p)\in Q$ is away from the vertex. Then:
\begin{align*}
    \delta_p(X)\geq \frac{44}{43}
\end{align*}
\end{proposition}
\begin{proof}
The result follows from applying Theorem~\ref{blow-up-of-surface-formula} to the flag consisting of the strict transform $S$ of a hyperplane in $\mathbb P^3$, the exceptional curve $F$ in $\widetilde S$.
    
We consider the linear system $-K_X-uS$. Its Zariski decomposition is then given by
\begin{align*}
P(u)=
\begin{cases}
(4-u)H - E         & \text{if}\ u\in[0,1],\\
 (6-3u)H + (u-2)E  & \text{if}\ u\in[1,2],
\end{cases}
\mbox{ and }
N(u)=
\begin{cases}
0                  & \text{if}\ u\in [0,1],\\
(u-1)\widetilde Q  & \text{if}\ u\in [1,2],
\end{cases}
\end{align*}
A direct computation gives 
\begin{align}\label{eq: S-of-plane}
    \frac{A_X(S)}{S_X(S)}=\frac{44}{23}.
\end{align}

We consider then the linear system
\begin{align*}
D=
    \sigma^{*}(P(u)|_{S})-vF =
    \begin{cases}
    (4-u)h-\sum_{i=1}^6 E_i - vF      & \text{if } u\in [0,1],\\
     (6-3u)h-(2-u)\sum_{i=1}^6 E_i -vF     & \text{if } u\in [1,2].
    \end{cases}
\end{align*}
Its Zariski decomposition for $u\in[0,1]$ is given by

\resizebox{.99\linewidth}{!}{
  \begin{minipage}{\linewidth}
\begin{align*}
P =
\begin{cases}
D                                       & \text{if}\ v\in [0,2-2u],\\
D-a\widetilde C                         & \text{if}\ v\in [2-2u,3-u],\\
D-a\widetilde C - b\sum_{i=1}^6 L_j  & \text{if}\ v\in [3-u,\frac{1}{4}(14-5u)].
\end{cases}
\mbox{ and }
N =
\begin{cases}
0                                        & \text{if}\ v\in [0,2-2u],\\
a\widetilde C                            & \text{if}\ v\in [2-2u,3-u],\\
a\widetilde C + b\sum_{i=1}^6 L_j     & \text{if}\ v\in [3-u,\frac{1}{4}(14-5u)].
\end{cases}
\end{align*}  \end{minipage}
}

where $a=\frac{1}{3}(v+2u-2)$ and $b=v-3+u$. For $u\in [1,2]$ it is given by

\resizebox{.99\linewidth}{!}{
  \begin{minipage}{\linewidth}
\begin{align*}
P =
\begin{cases}
D - a\widetilde C                             & \text{if}\ v\in [0,4-2u],\\
D -a\widetilde C - b \sum_{j=1}^6 L_j      & \text{if}\ v\in [4-2u,\frac{1}{4}(18-9u)].
\end{cases}
\mbox{ and }
N =
\begin{cases}
a\widetilde C                                        & \text{if}\ v\in [0,4-2u],\\
a\widetilde C + b \sum_{j=1}^6 L_j                            & \text{if}\ v\in [4-2u,\frac{1}{4}(18-9u)].
\end{cases}\\
\end{align*}  \end{minipage}
}

where $a=\frac{v}{3}$ and $b = v-4+2u$.
Hence, for $u\in[0,1]$ the volume of the divisor $D$ is
\begin{align*}
\mathrm{vol}(D) = P^2 =
\begin{cases}
u^2 - v^2 - 8u + 10                  & \text{if}\ v\in [0,2-2u],\\
\frac{1}{3}(7u^2 + 4uv - 2v^2 - 32u - 4v + 34)                     & \text{if}\ v\in [2-2u,3-u],\\
\frac{1}{3}(5u + 4v - 14)^2  & \text{if}\ v\in [3-u,\frac{1}{4}(14-5u)]
\end{cases}
\end{align*}

and for $u\in[1,2]$
\begin{align*}
\mathrm{vol}(D) = P^2 =
\begin{cases}
u^2 - v^2 - 8u + 10                  & \text{if}\ v\in [0,4-2u],\\
\frac{1}{3}(7u^2 + 4uv - 2v^2 - 32u - 4v + 34)                     & \text{if}\ v\in [4-2u,\frac{1}{4}(18-9u)].
\end{cases}
\end{align*}

We note that for $u\in [1,2]$ the contribution of the negative part in \eqref{Rmk1.7.32-surface-to-curve} is $\mathrm{ord}_F(\sigma^* N(u)\vert_S) = \mathrm{ord}_F((u-1)(\widetilde C+F)) = u-1$. So the value can be computed 
\begin{align}\label{eq: not vertex srf-to-crv}
    \frac{A_S(F)}{S(V^S_{\bullet,\bullet};F)} = \frac{8}{7}.
\end{align}
We now compute $S(W_{\bullet,\bullet,\bullet}^{S,F};q)$. Since $\widetilde S$ is smooth, the different $\Delta_q$ is trivial for any point $q$, while the value of $F_q(W_{\bullet,\bullet,\bullet}^{S, F})$ depends on the position of $q$ in $F$. We split thus into following three cases:
\begin{itemize}
    \item $q \notin \widetilde C\cup \bigcup_{j=1}^6 L_j$, so that $\mathrm{ord}_q(N_{\tilde{S}}'(u)\vert_{F}+\tilde N(u,v)\vert_{F}) = 0$ and $F_q=0$. And one has:
    \begin{align*}
        \frac{1-\mathrm{ord}_q\Delta_{\tilde{Z}}}{S(W_{\bullet,\bullet,\bullet}^{Y,\tilde{Z}};q)} = \frac{22}{15}.
    \end{align*}

    \item $q=\widetilde C \cap F$ so that 
    \begin{align*}
        \mathrm{ord}_q(N_{\tilde{S}}'(u)\vert_{F}+\tilde N(u,v)\vert_{F}) =\begin{cases}
            \frac{1}{3}(v+2u-2)    &\text{ if } u\in[0,1] \text{ and } v\in[2-2u, \frac{1}{4}(14-5u)],\\
            u-1+\frac{v}{3}       &\text{ if } u\in[1,2] \text{ and } v\in[0, \frac{1}{4}(18-9u)],\\
            0             &\text{ otherwise.}
        \end{cases}
    \end{align*}

    From which, one can compute:
    \begin{align*}
        F_q(W_{\bullet,\bullet,\bullet}^{S, F}) = \frac{13}{44}\quad  \mbox{ and }\quad
        \frac{1-\mathrm{ord}_q\Delta_{\tilde{Z}}}{S(W_{\bullet,\bullet,\bullet}^{Y,\tilde{Z}};q)} = \frac{44}{43}.
    \end{align*}
    
    \item $q= F\cap L_j$ for some $j=1,...,6$ so that
    \begin{align*}
        \mathrm{ord}_q(N_{\tilde{S}}'(u)\vert_{F}+\tilde N(u,v)\vert_{F}) =\begin{cases}
            v+u-3         &\text{ if } u\in[0,1] \text{ and } v\in[3-u, \frac{1}{4}(14-5u)],\\
            v-4+2u        &\text{ if } u\in[1,2] \text{ and } v\in[4-2u, \frac{1}{4}(18-9u)],\\
            0             &\text{ otherwise.}
        \end{cases}
    \end{align*}
    From which, one can compute:
    \begin{align*}
        F_q(W_{\bullet,\bullet,\bullet}^{S, F}) = \frac{1}{66}\quad \mbox{ and }\quad
        \frac{1-\mathrm{ord}_q\Delta_{\tilde{Z}}}{S(W_{\bullet,\bullet,\bullet}^{Y,\tilde{Z}};q)} = \frac{33}{23}.
    \end{align*}
\end{itemize}
Therefore,
\begin{align}\label{eq:not vertex crv-to-pnt}
\min_{q\in F}\frac{1-\mathrm{ord}_q\Delta_{F}}{S(W_{\bullet,\bullet,\bullet}^{S,F};q)}=\min\left\{\frac{22}{15}, \frac{44}{43}, \frac{33}{23} \right\}=\frac{44}{43}.
\end{align}

Finally, by combining the Equations~\eqref{eq: S-of-plane}, \eqref{eq: not vertex srf-to-crv}, and \eqref{eq:not vertex crv-to-pnt}, we get
\begin{align*}
    \delta_p(X)\geq \min \left\{\frac{44}{37}, \frac{8}{7}, \frac{44}{43} \right\}= \frac{44}{43}.
\end{align*}

\end{proof}

\subsection{ Estimate of \texorpdfstring{$\delta_p$ for a point $p$ off $E$ and $\tilde Q$}.}\phantom.

In this section, we estimate $\delta_p(X)$ for a point $p\in X \setminus (E \cup \tilde{Q})$. Roughly speaking, we consider the flag given by the general hyperplane section of $V_3$ containing $\beta(p)$ and the curve given by its tangent hyperplane section. The precise flag depends though on the singularity of the latter.

\begin{lemma}\label{S-of-H_V3}
    Let $S$ be the strict transform of a hyperplane section of $V_3$ not containing the singular point of $\beta(\widetilde Q)$.
    Then 
\begin{align*}
    S_{X}(S) = \frac{14}{33}.
\end{align*}
\end{lemma}

\begin{proof}
The linear system $-K_X-uS$ can be written as
\[
-K_X-uS=\Big(2-\frac{3}{2}u\Big)\tilde{Q}+\Big(1-\frac{u}{2}\Big)E = (4-3u)H+(u-1)E.
\]
Thus its pseudoeffective threshold is $\tau(u)=\frac{4}{3}$ and its Zariski decomposition is given by:
\begin{align*}
P(u)=
\begin{cases}
(4-3u)H+(u-1)E & \text{if}\ u\in [0,1],\\
(4-3u)H & \text{if}\ u\in [1,\frac{4}{3}].
\end{cases}
\mbox{ and } N(u)=
\begin{cases}
0 & \text{if}\ u\in [0,1],\\
(u-1)E & \text{if}\ u\in [1,\frac{4}{3}].
\end{cases}
\end{align*}

Therefore, 
$\mathrm{vol}(-K_X-uS)=
\begin{cases}
22-36u+18u^2-3u^3& \text{if}\ u\in [0,1],\\
64-144u+108u^2-27u^3 & \text{if}\ u\in [1,\frac{4}{3}].
\end{cases}
$

Hence,
\begin{align}\label{3:3-fold-surface}
    S_{X}(S)&=\frac{1}{(-K_X)^3}\int_{0}^{\tau(S)}\mathrm{vol}(-K_X-uS)du=\frac{14}{33}.
\end{align}
\end{proof}

 We consider a hyperplane section $S$ of $V_3$ containing the point $\beta(p)$ and not containing the point $\beta(\widetilde Q)$, so that $S$ is a smooth cubic surface. We study the singularities of its tangent hyperplane section, because the relevant flag we use to estimate $\delta_p$ depends on them.

For an appropriate choice of coordinates $\beta(p)=(0,0,0,0)\in \mathbb C^4_{x,y,z,t}$ in a chart of $\mathbb P^{4}$ and the surface $S$ is given by:
\begin{align*}
S=
\begin{cases}
x + f_{2}(x,y,z,t)+f_{3}(x,y,z,t)= 0,\\
y=0.
\end{cases}
\end{align*}
where $f_{2}$ (respectively $f_{3}$) is a homogeneous polynomial of degree 2 (respectively of degree 3).
By considering a suitable change of variables
we might assume that no monomials containing $x$ appear in the expression of $f_{2}$ and so we have:
\begin{align*}
\mathrm{rk}(f_{2}\vert_{(y=0)})\in \{ 0,1,2\}.
\end{align*}
The tangent hyperplane section of $S$ is the curve $C$ given by:
\begin{align*}
(x=0) \cap S =
\begin{cases}
x=y=0,\\
f_{2}(0,z,t) + f_{3}(0,0,z,t)=0.
\end{cases}
\end{align*}

Therefore, the curve $C$ consists of
\begin{itemize}
\item a rational curve with a node at $\beta(p)$ if $\mathrm{rk}(f_{2})=2$;

\item a rational curve with a cusp at $\beta(p)$ if $\mathrm{rk}(f_{2})=1$;

\item three lines intersecting at $\beta(p)$ if $\mathrm{rk}(f_{2})=0$.
\end{itemize}
In each of these cases we use a different flag.

Since we are assuming that $S$ does not contain the point $\beta(\widetilde Q)$, the surface $S$ is isomorphic to its strict transform in $X$, and so is $C$. In what follows we slightly abuse notation and use the symbols $S$ and $C$ for the strict transforms as well.

\subsubsection{ Nodal curve}\phantom.

Suppose the point $p$ on $X$ is such that the curve $C$ on $V_3$ is a curve with a node at $\beta(p)$. In order to estimate $\delta_p$, we make use of Theorem \ref{blow-up-of-surface-formula}. Let $\sigma\colon \widehat{S}\to S$ be the blow-up of $S$ in $p$ with exceptional curve $G$. We denote by $\widehat C$ the strict transform of $C$ in $\widehat S$. We have the following intersection numbers:
\begin{align*}
G^2=-1,\quad G\cdot \widehat C=2, \quad \widehat C^2 = -1.
\end{align*}

\begin{proposition} \label{prop: nodal}
Suppose that $p \in X \backslash (\tilde{Q} \cup E)$ is such that $\beta(p)$ is the node of the tangent hyperplane section to the general hyperplane section of $V_3$ containing $\beta(p)$, then
\begin{align*}
    \delta_p(X) \geq \frac{176}{161}.
\end{align*}
\end{proposition}

\begin{proof}
We apply Theorem \ref{blow-up-of-surface-formula} to the flag consisting of $p$, the exceptional curve $G$ and the strict transform of the general hyperplane section of $V_3$ through $\beta(p)$.
For this, we consider the linear system
\begin{align*}
    \sigma^{*}(P(u)\vert_{S})-vG =
    \begin{cases}
    (2-u)\widehat{C}+(4-2u-v)G & \text{if } u\in [0,1],\\
    (4-3u)\widehat{C}+(8-6u-v)G & \text{if } u\in [1,\frac{4}{3}].
    \end{cases}
\end{align*}

Its Zariski decomposition is given by
\begin{align*}
\tilde P(u,v)=
\begin{cases}
(2-u)\widehat{C} + (4-2u-v)G & \text{if } u\in [0,1]\    v\in [0,3-\frac{3u}{2}];\\
(4-2u-v)(2\widehat{C}+G)        & \text{if } u\in [0,1]\     v\in [3-\frac{3u}{2}, 4-2u] \\
(4-3u)\widehat{C} + (8-6u-v)G   & \text{if } u\in [1,\frac{4}{3}]\ v\in [0,6-\frac{9u}{2}] \\
(8-6u-v)(2\widehat{C}+G)        & \text{if } u\in [1,\frac{4}{3}]\ v\in [6-\frac{9u}{2},8-6u].
\end{cases}
\end{align*}
and by
\begin{align*}
\tilde N(u,v)=
\begin{cases}
0 & \text{if } u\in [0,1] \   v\in [0,3-\frac{3u}{2}];\\
(2v+3u-6)\widehat{C}      & \text{if } u\in [0,1] \   v\in [3-\frac{3u}{2}, 4-2u] \\
0   & \text{if } u\in [1,\frac{4}{3}]\ v\in [0,6-\frac{9u}{2}] \\
(2v+9u-12)\widehat{C}     & \text{if } u\in [1,\frac{4}{3}]\ v\in [6-\frac{9u}{2},8-6u].
\end{cases}
\end{align*}
Its volume can be directly computed to be
\begin{align}
   \mathrm{vol}(\sigma^{*}(P(u)\vert_{S})-vG)= 
    {\begin{cases}
 3u^2-v^2-12u+12 & \text{if } u\in [0,1],\ v\in [0,3-\frac{3u}{2}],\\
 12u^2+12uv+3v^2-48u-24v+48 & \text{if } u\in [0,1],\ v\in [3-\frac{3u}{2},4-2u],\\
 27u^2-v^2-72u+48 & \text{if } u\in [1,\frac{4}{3}],\ v\in [0,6-\frac{9u}{2}],\\
 108u^2+36uv+3v^2-288u-48v+192  & \text{if } u\in [1,\frac{4}{3}],\ v\in [6-\frac{9u}{2},8-6u].\\
    \end{cases}}
\end{align}

We note that \[\mathrm{ord}_{p}N(u)\vert_S=
\begin{cases}
    0 &\text{if}\, u \in [0,1],\\
    \mathrm{ord}_{p}(u-1)E\vert_{S} &\text{if}\, u \in [1,\frac{4}{3}],
\end{cases}\] and therefore $\mathrm{ord}_{p}N(u)\vert_S=0$ since $p$ is not in $E$ by assumption. Thus,
\begin{align}\label{nodal: srf-to-crv}
    S(V_{\bullet,\bullet}^{S};G)=\frac{161}{88}.
\end{align}
Since $A_{S}(G)=1+\mathrm{ord}_G(K_{\widehat{S}}-\sigma^*(K_S))=2$, we have that $\frac{A_{S}(G)}{S(V_{\bullet,\bullet}^{S};G)}=\frac{176}{161}$.

Next, we compute $S(W_{\bullet,\bullet,\bullet}^{S,G};q)$. Straightforward computations using the intersection numbers gives us the first summand in \eqref{Rmk1.7.32-curve-to-point} 
\begin{align*} 
 \frac{3}{(-K_X)^3}\int_0^{\tau}\int_0^{\tilde{t}(u)}(\tilde{P}(u,v) \cdot G)^2dvdu=
 \begin{cases}
\frac{135}{176} &\text{if} u\in [0,1],\\
\frac{3}{176} &\text{if} u\in [1,\frac{4}{3}].
\end{cases}
\end{align*}
For $u \in [0,1]$ since $N_S(u)=0$, we have that $N'_{\tilde{S}}(u)=0$. When $u \in [1,\frac{4}{3}]$, $N_{\tilde{S}}(u)=(u-1)\widetilde{E \vert_S}$, where $\widetilde{E \vert_S}$ is the strict transform of the curve $E\vert_S$ on $\widehat{S}$. Since by assumption $ p\notin E$,  we have $N_{\tilde{S}}(u) \vert_G=0$. We have different cases depending on the position of the point $q$.

If $q \in G \cap \widehat{C}$ 
    \begin{align*}
 F_q(W_{\bullet,\bullet,\bullet}^{S, G})=
 \begin{cases}
 0 &\text{if}  u\in [0,1], v \in [0, 3-\frac{3u}{2}],\\
 \frac{45}{352}&\text{if} u\in [0,1], v \in [ 3-\frac{3u}{2},4-2u], \\
0&\text{if} u\in [1,\frac{4}{3}], v \in [0,6-\frac{9u}{2}],\\
 \frac{1}{352} &\text{if} u\in [1,\frac{4}{3}], v \in [6-\frac{9u}{2},8-6u].
 \end{cases}
    \end{align*}
and 

If $q \in G \backslash{\widehat {C}}$
\[
F_q(W_{\bullet,\bullet,\bullet}^{S,G})=0.
\]

The value in \eqref{Rmk1.7.32-curve-to-point} is then given by:
\begin{align*}
    S(W_{\bullet,\bullet,\bullet}^{S,G};q)=\frac{161}{176} \mbox{ when } q \in G \cap \widehat{C} \mbox{ and}\\
    S(W_{\bullet,\bullet,\bullet}^{S,G};q)=\frac{69}{88} \mbox{ when } q \in G \backslash{\widehat{C}}
\end{align*}
Since the surface $\widehat S$ is smooth, the different $\Delta_G$ is trivial and we get:  
\begin{align}\label{nodal: crv-to-pnt}
\mathrm{min}_{q\in G}\frac{1-\mathrm{ord}_q\Delta_{G}}{S(W_{\bullet,\bullet,\bullet}^{S,G};P)}=\frac{176}{161}.
\end{align}

In conclusion, combining Lemma~\ref{S-of-H_V3} and Equations \eqref{nodal: srf-to-crv} and \eqref{nodal: crv-to-pnt} we get 
\[
\delta_p(X) \geq \mathrm{min} \left\{\frac{176}{161},\frac{176}{161},\frac{33}{14}\right\}=\frac{176}{161}.
\]
\end{proof}

\subsubsection{ Cuspidal Curve}\phantom.

Suppose the point $p$ on $X$ is such that $C \subset S$ is cuspidal at the point $\beta(p)$. Similar to the previous subsection, we use Theorem \ref{blow-up-of-surface-formula} to obtain an estimate to $\delta_p(X)$.

Let $\sigma:\widehat{S}\to S$ be the $(2,3)$-weighted blow up
of $S$ at the point $p$ with exceptional divisor $G$. The strict transform $\widehat{C}$ of $C$ in $\widehat S$ intersects the exceptional curve $G$ in one regular point.
The following hold:
\begin{align*}
\widehat{C}=\sigma^*(C)-6G,\quad\quad K_{\widehat{S}}=\sigma^*(K_S)+4G,\mbox{ and}\\
G^{2}=-\frac{1}{6},\qquad \widehat C\cdot G=1,\qquad \widehat C^{2}=-3.
\end{align*}
We note that $G$ has two singular points, we denote by $p_0$ the one of type $\frac{1}{2}(1,1)$ and by $p_1$ the one of type $\frac{1}{3}(1,1)$. In particular, the different $\Delta_G$ defined by:
\begin{align*}
(K_{\widehat S} + G)\vert_G = K_{G} + \Delta_G\quad \mbox{ is given by }\quad \Delta_G = \frac{1}{2}p_0 + \frac{2}{3}p_1.
\end{align*}

\begin{proposition} \label{prop: cusp}
Suppose the point $p \in X \backslash (\tilde{Q} \cup E)$ is a cusp of the
tangent hyperplane section to the general hyperplane section of $V_3$ containing $\beta(p)$, then 
\begin{align*}
    \delta_p(X) \geq \frac{220}{207}.
\end{align*}
\end{proposition}

\begin{proof}
We apply Theorem \ref{blow-up-of-surface-formula} to the flag consisting of $p$, the exceptional curve $G$ and the strict transform of the general hyperplane section of $V_3$ through $\beta(p)$.

We start by computing $S(V_{\bullet,\bullet}^{\widehat{S}},G)$. We consider the linear system 
\begin{align*}
    \sigma^*(P(u)\vert_{S})-vG=
    \begin{cases}
    (2-u)\widehat{C}+(12-6u-v)G &\text{if } u\ \in[0,1],\\
    (4-3u)\widehat{C}+(24-18u-v)G&\text{if } u\ \in[1,\frac{4}{3}].
    \end{cases}
\end{align*}

Its Zariski decomposition has positive part given by
\begin{align*}
\tilde P(u,v) =
\begin{cases}
(2-u)\widehat C + (12-6u-v)G &\text{ if } u\in [0,1], v\in [0,6-3u];\\
(12-6u-v)(\frac{1}{3}\widehat C + G) &\text{ if } u\in [0,1], v\in [6-3u,12-6u];\\
(4-3u)\widehat C + (24-18u-v)G &\text{ if } u\in [1,\frac{4}{3}], v\in [0,12-9u];\\
(8-6u-\frac{v}{3})(\widehat C + 3G) &\text{ if } u\in [1,\frac{4}{3}], v\in [12-9u, 24-18u].
\end{cases}
\end{align*}
and negative given by:
\begin{align*}
\tilde N(u,v) =
\begin{cases}
0 &\text{ if } u\in [0,1], v\in [0,6-3u];\\
(u-2+\frac{v}{3})\widehat C &\text{ if } u\in [0,1], v\in [6-3u,12-6u];\\
0 &\text{ if } u\in [1,\frac{4}{3}], v\in [0,12-9u];\\
(\frac{v}{3}+3u-4)\widehat C &\text{ if } u\in [1,\frac{4}{3}], v\in [12-9u, 24-18u].
\end{cases}
\end{align*}
Note that $N_S(u)\vert_S=0$ for $u \in [0,1]$ and  $\mathrm{ord}_G(N_S(u)\vert_S)=\mathrm{ord}_G((u-1)E\vert_S)=0$ for $u \in [1,\frac{4}{3}]$ since by assumption $p \notin E$. Therefore the value in Equation~\eqref{Rmk1.7.32-surface-to-curve} 
\begin{align}\label{cuspidal: srf-to-crv}
    S(V_{\bullet,\bullet}^{S},G)=\frac{207}{44},\quad \mbox{ and thus }\quad \frac{A_{S}(G)}{S(V_{\bullet,\bullet}^{S},G)}=\frac{220}{207},
\end{align}  since $A_{S}(G)=1+\mathrm{ord}_G(K_{S}-\sigma^*(K_S))=5$.

We now compute $S(W_{\bullet,\bullet,\bullet}^{S,G};q)$ for various points $q\in G$. To compute the value in formula \eqref{Rmk1.7.32-curve-to-point} we notice that the first term is independent of the position of $q\in G$, while $F_{q}:=F_q(W_{\bullet,\bullet,\bullet}^{S,G})$ varies, so we split in cases.
We notice that $\mathrm{ord}_q(N_S'(u)\vert_{G})=0$ for any point $q\in G$, since $N_{S}(u)$ is a multiple of $E$ and $p\not\in E$ by assumption. Also, $\tilde N(u,v)$ is a multiple of $\widehat C$, hence $F_{q}\not=0$ only for $q = G\cap\widehat C$.
We have $S(W^{S,G}_{\bullet,\bullet,\bullet}; q) = \frac{23}{88} + F_{q}$ 
and we get the cases: 
\begin{itemize}
\item $q=p_{0}$, so $\mathrm{ord}_{p_0}(\Delta_G)=\frac{1}{2}$ and
\begin{align*}
\frac{1-\mathrm{ord}_q(\Delta_G)}{S(W^{ S,G}_{\bullet,\bullet,\bullet}; q)} = \bigg(1-\frac{1}{2}\bigg)\cdot \frac{88}{23} = \frac{44}{23};
\end{align*}

\item $q=p_{1}$, so $\mathrm{ord}_{p_1}(\Delta_G)=\frac{2}{3}$ and
\begin{align*}
\frac{1-\mathrm{ord}_q(\Delta_G)}{S(W^{S,G}_{\bullet,\bullet,\bullet}; q)} = \bigg(1-\frac{2}{3}\bigg)\cdot \frac{88}{23} = \frac{88}{69};
\end{align*}

\item $q=\widehat C\cap G$, so $\mathrm{ord}_{q}(\Delta_G)=0,\ F_{q}=\frac{23}{88}$ 
and
\begin{align*}
\frac{1-\mathrm{ord}_q(\Delta_G)}{S(W^{ S,G}_{\bullet,\bullet,\bullet}; q)} = \frac{1}{\frac{23}{88} + \frac{23}{88}} = \frac{44}{23};
\end{align*}

\item $q\not\in\{p_{0},p_{1},\widehat C\cap G\}$, so $\mathrm{ord}_{p_q}(\Delta_G)=0$ and
\begin{align*}
\frac{1-\mathrm{ord}_q(\Delta_G)}{S(W^{ S,G}_{\bullet,\bullet,\bullet}; q)} = \frac{88}{23}.
\end{align*}
\end{itemize}
Therefore,
\begin{align}\label{cuspidal: crv-to-pnt}
\min_{q\in G}\frac{1-\mathrm{ord}_q\Delta_{G}}{S(W_{\bullet,\bullet,\bullet}^{S,G};q)}=\min\left\{\frac{88}{23},\frac{44}{23},\frac{88}{69},\frac{44}{23}\right\}=\frac{88}{69}.
\end{align}

In conclusion, by Lemma~\ref{S-of-H_V3} and Equations~\eqref{cuspidal: srf-to-crv} and \eqref{cuspidal: crv-to-pnt} we have:
\[
\delta_p(X) \geq \mathrm{min} \left\{\frac{33}{14},\frac{220}{207},\frac{88}{69}\right\}=\frac{220}{207}.
\]
\end{proof}

\subsubsection{Three lines}\phantom.

Suppose the point $p \in X$ is such that the curve $C \subset S$ containing $\beta(p)$ is a union of 3 lines that intersect at $\beta(p)$. Then, unlike the previous 2 cases, blowing up the surface $S$ in $X$ does not prove useful in giving a good estimate to $\delta_p(X)$ and therefore, we will use the notion of infinitesimal flags over $X$. 

Let $\pi:\hat{X} \to X$ be the blow up of the 3-fold $X$ at the point $p$, with the exceptional divisor given by $G$ and the strict transform of the surface $S$ given by $\hat{S}$.
Since $-K_{X}=\beta^*(2S)-\tilde{Q}$ and $-K_{\hat{X}}=\pi^*(-K_X)-2G$, the divisor
\begin{align}\label{eq: canonical blow-up cubic cone}
\pi^*(-K_X)-uG=\frac{4}{3}\hat{S}+\frac{1}{3}\hat{E}+(4-u)G
\end{align}
where we also use $\hat{S}=\pi^*(S)-3G$ and $\tilde{Q}=\frac{2}{3}S-\frac{1}{3}E$.

\begin{lemma}\label{lemma:bigness of divisor}
The pseudo-effective threshold $\tau$ of the linear system $\pi^*(-K_X)-uG$ is $\tau=4$.
\end{lemma}
\begin{proof}
From Equation~\eqref{eq: canonical blow-up cubic cone} we clearly have that $\tau\geq 4$. In order to prove the equality we show that the divisor $4\hat S + \hat E$ is not big. For this, let $\gamma\colon \hat X\to \Bl_{\alpha(p)}\mathbb P^3$ be the divisorial contraction of $\hat E$. Since the pushforward of a big divisor along a birational morphism is big, in order to show the claim it is enough to show that $\gamma_{*}\hat S$ is not big. For this, notice that $\Bl_{\alpha(p)}\mathbb P^{3}$ is the resolution of indeterminacy of the projection from $\alpha(p)$ and is a conic bundle $h\colon\Bl_{\alpha(p)}\mathbb P^{3}\to \mathbb P^2$ which contracts $\gamma(\hat S)$ to an elliptic curve. In particular $\gamma(\hat S) = h^*\mathcal O_{\mathbb P^2}(3)$ is not big. The claim is proven.
\end{proof}

\begin{proposition} \label{prop: three lines}
Suppose $p \in X \backslash (\tilde{Q} \cup E)$ is such that $p$ is the Eckardt point of curve $C$ given by the tangent hyperplane section to the general hyperplane section of $V_3$ containing $\beta(p)$. Then \[\delta_p(X) = \frac{22}{17}\] and it is computed by the exceptional divisor $G$ corresponding to the ordinary blowup of $X$ at $p$.
\end{proposition}

\begin{proof}
  By \cite[Corollary 4.18 (2)]{Kentorank3deg28}, we have
\begin{equation} \label{eq:AZmain2}
 \frac{A_X(G)}{S_X(G)} \geq \delta_p(X) \geq  \min \bigg\{\frac{A_X(G)}{S_X(G)},
\inf_{\substack{q \in G}} \delta_q(G,\Delta_G;V^{G}_{\bullet,\bullet}) \bigg\},
\end{equation} where the infimum runs over all points $q \in G$.

We first compute the left hand side of inequality \ref{eq:AZmain2} and prove that the right hand side is bounded below by $\frac{A_X(G)}{S_X(G)}$. From the proof of Lemma \ref{lemma:bigness of divisor}, we know that $\hat{S}$ is a cone over an elliptic curve. Let $L$ be the class of a line in $\hat S$, then:
\begin{align*}
G\cdot L = 1 \quad \hat E\cdot L = 2 \quad \mbox{and}\quad \hat S\cdot L = -2.
\end{align*}

Moreover,
\begin{align*}
 \hat S^2\cdot E = 6, \quad  \hat S\cdot G^{2} = -3,\quad  \hat S^2 \cdot G = 9, \quad G^2\cdot \hat E = G \cdot \hat E^2 = 0,\\
 \hat E^{3} = -30\quad  \hat S\cdot E^{2} = 12,\quad \hat S_{x}^{3} = -24, \quad \hat S \cdot \hat E \cdot G = 0,\quad G^{3} = 1. 
\end{align*}

Let $P(u)$ and $N(u)$ be the positive and negative part of $\pi^{*}(-K_X) - uG$. We have:
\begin{align*}
P(u)=
\begin{cases}
\frac{1}{3}\hat E + \frac{4}{3}\hat S + (4-u)G & \text{if}\ u\in[0,2],\\
\frac{1}{3}\hat E +\left( \frac{7}{3}-\frac{u}{2} \right)\hat S + (4-u)G & \text{if}\ u\in[2,4],
\end{cases}
\mbox{ and }
N(u)=
\begin{cases}
0 & \text{if}\ u\in [0,2],\\
\frac{(u-2)}{2}\hat S & \text{if}\ u\in [2,4],
\end{cases}
\end{align*}

and since $-K_{\hat{X}}=\pi^*(-K_X)-2G$, we have that $A_X(G)=3$ and $S_X(G)=\frac{51}{22}$ 
, so that
\[
\frac{A_X(G)}{S_X(G)}=\frac{22}{17}.
\]

We now estimate $\inf_{\substack{q \in G}} \delta_q(G,\Delta_G;V^{G}_{\bullet,\bullet})$. 
For every point $q \in G$, we choose the flag $q \in L \subset G$, where $L$ is a line in $G$ intersecting $\hat S\vert_G$ transversely. Then,  by \cite[Theorem~3.2]{AZ21}
\begin{align*}
\delta_q (G,\Delta_G; W^G_{\bullet,\bullet})
\geq \mathrm{min} \left\{ \frac{1}{S(W_{\bullet,\bullet}^G;L)},\frac{1-\mathrm{ord}_q{\Delta_L}}{S(W_{\bullet,\bullet,\bullet}^{G,L};q)} \right\}.
\end{align*}
 Let $P(u,v)$ and $N(u,v)$ be the positive and negative part of $P(u)\vert_G-vL$. These are given by 
\begin{align*}
    P(u,v)=
    \begin{cases}
    (u-v)L &\text{if}\ u\in [0,2],\ v\in [0,u],\\
    \Big(\frac{6-u}{2}-v\Big)L &\text{if}\ u\in [2,4],\ v\in [0, \frac{6-u}{2}],
    \end{cases}\quad
    \text{ and}\quad  N(u,v) = 0.
\end{align*}

 Notice that $\mathrm{ord}_L(N(u)\vert_{G})=0$ since $\hat S\vert_G$ is not supported on $L$. Then, 
\[
\frac{1}{S(W_{\bullet,\bullet}^G;L)}=\frac{44}{23}.
\]
Let $Z$ be the elliptic curve $\hat S\vert_G$. Then,
\begin{align*}
\mathrm{ord}_{q}(N'_{ G}(u)\vert_{L}+N(u,v)\vert_L)&=\mathrm{ord}_{q}(N'_{ G}(u)\vert_{L})\\ 
&=\mathrm{ord}_{q}\bigg(\frac{u-2}{2}Z\vert_L\bigg)\\
&=\begin{cases}
0 & \text{if}\ q \not \in Z\vert_L  ,\\
\frac{u-2}{2}  & \text{otherwise}.
\end{cases}
\end{align*}
Then,
\[
S(W^{G,L}_{\bullet,\bullet,\bullet};q)= 
\begin{cases}
\frac{23}{44} & \text{if}\ q \not \in Z\vert_L  ,\\
\frac{17}{22} & \text{if}\ q \in Z\vert_L .
\end{cases}
\]
Hence,
\[\inf_{\substack{q \in G}} \delta_q(G,\Delta_G;V^{G}_{\bullet,\bullet}) \geq \min \bigg\{\frac{44}{23}, \min \Big\{\frac{44}{23},\frac{22}{17} \Big\} \bigg\}=\frac{22}{17}.
\]
The claim follows.
\end{proof}

\subsection{ Estimate of \texorpdfstring{$\delta_p$ for a point $p$ in $E$}.}

We now estimate $\delta_p(X)$ where  $p\in E$. Let $H\subseteq \mathbb P^{3}$ be a general hyperplane containing $\alpha(p)$. Then, recall that $H$ intersects the curve $\mathscr C$ in six points, which we denote $p_{1}:=\alpha(p), p_2,\ldots, p_6$, lying on the conic $C:=Q\cap H$. Let $S$ be the strict transform of $H$. The morphism $S\to H$ is the blow-up of $H\simeq \bbP^{2}$ in the six points $p_{1}, p_2,\ldots, p_6$ and we denote by $E_{i}$ the associated exceptional divisors. Let $l$ be the strict transform of a line in $H$. 
\begin{proposition} \label{prop: p in E}
If $p \in  E \cap \tilde Q$,  then
\begin{align*}
    \delta_p(X) \geq \frac{132}{131}.
\end{align*}
If $p \in  E \setminus \tilde Q$, then
\begin{align*}
    \delta_p(X) \geq \frac{66}{65}.
\end{align*}
\end{proposition}
\begin{proof}

We apply Theorem~\ref{delta estimate} to the flag:
\begin{align*}
    p \in E_1 \subset S \subset X
\end{align*}
To compute $S_X(S)$ we consider the linear system $-K_{X}-uS$ for $u \in \mathbb{R}_{\geq 0}$ which, in terms of the generators of $\mathrm{Eff}(X)$, is given by 
\begin{align*}
    \bigg(2-\frac{u}{2}\bigg)\tilde Q +\bigg(1-\frac{u}{2}\bigg)E. 
    \end{align*}
    Hence its pseudo-effective threshold is $\tau=2$. Consider the Zariski decomposition of $-K_X-uS$:
\begin{align*}
P(u) =
\begin{cases}
(4-u)H - E & \text{if}\ u\in [0,1],\\
(6-3u)H + (u-2)E & \text{if}\ u\in [1,2],
\end{cases}
\mbox{ and }
N(u) =
\begin{cases}
0\ & \text{if}\ u\in[0,1],\\
(u-1)\tilde Q\ & \text{if}\ u\in[1,2].
\end{cases}
\end{align*}
Recall that this is the same as in Proposition \ref{prop: p not vertex}, from which we have that $S_X(S)=\frac{23}{44}$ \eqref{eq: S-of-plane}. We now compute the value $S(V^{S}_{\bullet,\bullet};E_1)$. Consider the linear system 
\begin{align*}
D&=P(u)\vert_{S} - vE_{1} =
\begin{cases}
(4-u)l- \sum_{i=i}^6E_i-vE_1 & \text{if}\ u\in [0,1],\\
(6-3u)l-(2-u)\sum_{i=i}^6E_i-vE_1 & \text{if}\ u\in [1,2]
\end{cases}
\end{align*}
 We denote by $L_{i,j}$ the strict transform of the line through the points $p_i,p_j$. Its Zariski decomposition for $u\in [0,1]$ is
\begin{align*}
P =
\begin{cases}
D & \text{if}\ v\in [0,2-2u],\\
D-a\tilde{C} & \text{if}\ v\in [2-2u,2-u],\\
D-a\tilde{ C}-b\sum_{j=2}^6L_{1,j} & \text{if}\ v\in [2-u,\frac{8-4u}{3}]
\end{cases}
\mbox{ and }
N =
\begin{cases}
0& \text{if}\ v\in [0,2-2u],\\
a\tilde{C} & \text{if}\ v\in [2-2u,2-u],\\
a\tilde{ C}+b\sum_{j=2}^6L_{1,j} & \text{if}\ v\in [2-u,\frac{8-4u}{3}],
\end{cases}
\end{align*} 

where $a=\frac{v}{2}+u-1$ and $b=v+u-2$ and for $u\in [1,2]$ is
\begin{align*}
P =
\begin{cases}
D-a\tilde{ C} & \text{if}\ v\in [0,2-u],\\
D-a\tilde{ C}-b\sum_{j=2}^6L_{1,j} & \text{if}\ v\in [2-u,\frac{8-4u}{3}]
\end{cases}
\mbox{ and }
N =
\begin{cases}
a\tilde{C} & \text{if}\ v\in [0,2-u],\\
a\tilde{ C}+b\sum_{j=2}^6L_{1,j} & \text{if}\ v\in [2-u,\frac{8-4u}{3}]
\end{cases}
\end{align*}
where $a=\frac{v}{2}$ and $b=v+u-2$. Hence, the volume of the divisor $D$ for $u\in [0,1]$ is 
\begin{align*}
\mathrm{vol}(D)= P^2=
\begin{cases}
u^2-v^2-8u-2v+10& \text{if}\ v\in [0,2-2u],\\
12-4v-12u-\frac{1}{2}v^2+2vu+3u^2 & \text{if}\ v\in [2-2u,2-u],\\
\frac{1}{2}(4u+3v-8)^2  & \text{if}\ v\in [2-u,\frac{8-4u}{3}],
\end{cases}
\end{align*}
and for $u\in [1,2]$ is
\begin{align*}
\mathrm{vol}(D)= P^2=
\begin{cases}
12-4v-12u-\frac{1}{2}v^2+2vu+3u^2 & \text{if}\ v\in [0,2-u],\\
\frac{1}{2}(4u+3v-8)^2  & \text{if}\ v\in [2-u,\frac{8-4u}{3}].
\end{cases}
\end{align*}

We note that $\tilde Q \vert_{S} = \tilde C$ which has no support on $E_1$. Hence the negative part $N_{S}(u)$ does not contribute in the formula \eqref{surface-to-curve} and we get:
\begin{align} \label{Flag3:surface-to-curve}
    S(V^{S}_{\bullet,\bullet};E_1) = \frac{65}{66}.
\end{align}

We now compute $S(W_{\bullet,\bullet,\bullet}^{S,E_1};p)$. Notice that the hyperplane $H$ can be chosen so that $p$ does not lie in any of the $L_{1,j}$. Then 
\[
\mathrm{ord}_p(N'_{S}(u)\vert_{E_1}+N\vert_{E_1})=\mathrm{ord}_p\bigg(\Big(u-1+\frac{v}{2}\Big)\tilde C\vert_{E_1}\bigg)=\begin{cases}
u-1+\frac{v}{2} & \text{if}\ p \in \tilde C,\\
0  & \text{if}\ p \not \in \tilde C,
\end{cases}
\]
since $E_1$ is transversal to $\tilde C$. By \eqref{Fp}, we have,
\[
F_p(W_{\bullet,\bullet,\bullet}^{S,E_1})=\begin{cases}
\frac{5}{33} & \text{if}\ p \in \tilde C,\\
0  & \text{if}\ p \not \in \tilde C.
\end{cases}
\]
By direct application of  \eqref{curve-to-point} we have
\begin{align}\label{Flag3:curve-to-point}
S(W_{\bullet,\bullet,\bullet}^{\widetilde H,E_1};p) = \begin{cases}
\frac{131}{132} & \text{if}\ p \in \tilde C,\\
\frac{37}{44}  & \text{if}\ p \not \in \tilde C.
\end{cases}
\end{align}
By Theorem~\ref{delta estimate} we have,
\begin{align}
\delta_p(X) \geq  \begin{cases}
\min \Big\{ \frac{44}{23}, \frac{66}{65}, \frac{132}{131} \Big\} = \frac{132}{131} & \text{if}\ p \in \tilde C,\\
\min \Big\{ \frac{44}{23}, \frac{66}{65}, \frac{44}{37}\Big\} = \frac{66}{65}  & \text{if}\ p \not \in \tilde C.
\end{cases}
\end{align}
\end{proof}

We can finally prove our main theorem:

\begin{theorem} \label{thm:main}
    Every smooth member of the Fano family 2.15, which is the blow-up of $\mathbb P^3$ in a curve given by the intersection of a quadric and a cubic, is K-stable.
    In particular,
    \begin{align*}
        \delta(X) \geq \frac{132}{131}.
    \end{align*}
\end{theorem}    
\begin{proof}
The local stability threshold is estimated for every point $p\in X$. In particular by Propositions~\ref{delta: p in Q}, \ref{prop: p vertex}, \ref{prop: p not vertex}, \ref{prop: nodal}, \ref{prop: cusp}, \ref{prop: three lines}, \ref{prop: p in E} one has:
\begin{align*}
    \delta(X) \geq \min \bigg\{\frac{8}{7}, \frac{11}{10}, \frac{8}{7}, \frac{176}{161},\frac{220}{207}, \frac{22}{17}, \frac{132}{131} \bigg\}= \frac{132}{131}.
\end{align*}
\end{proof}

\newpage

\newcommand{\etalchar}[1]{$^{#1}$}


\end{document}